\documentclass[12pt, reqno]{amsart}

\usepackage[a4paper, centering, total={170mm,240mm}]{geometry}

\usepackage{amssymb,latexsym}
\usepackage{blindtext}
\usepackage[backgroundcolor=olive, linecolor=olive, textsize=tiny]{todonotes}
\usepackage{esint,dsfont}

\usepackage{palatino}

\usepackage[colorlinks=true, linkcolor=blue, citecolor=purple, urlcolor=blue]{hyperref}

\def\N{\mathbb{N}}
\def\C{\mathbb{C}}

\newtheorem{prop}{\bf Proposition}[section]
\newtheorem{thm}[prop]{\bf Theorem}
\newtheorem{cor}[prop]{\bf Corollary}
\newtheorem{lem}[prop]{\bf Lemma}
\newtheorem{rmk}[prop]{\it Remark}

\begin{document}

\title[Geometric Property (T) and Kazhdan projections]{{\bf\Large Geometric Property (T) and Kazhdan projections}}

\author[I. Vergara]{Ignacio Vergara}
\address{Departamento de Matem\'atica y Ciencia de la Computaci\'on, Universidad de Santiago de Chile, Las Sophoras 173, Estaci\'on Central 9170020, Chile}

\email{ign.vergara.s@gmail.com}
\thanks{This work is supported by the FONDECYT project 3230024.}

\makeatletter
\@namedef{subjclassname@2020}{%
  \textup{2020} Mathematics Subject Classification}
\makeatother

%


\begin{abstract}
We characterise Geometric Property (T) by the existence of a certain projection in the maximal uniform Roe algebra $C_{u,\max}^*(X)$, extending the notion of Kazhdan projection for groups to the realm of metric spaces. We also describe this projection in terms of the decomposition of the metric space into coarsely connected components.
\end{abstract}


\begingroup
\def\uppercasenonmath#1{} 
\let\MakeUppercase\relax 
\maketitle
\endgroup

\section{{\bf Introduction}}

\subsection{Kazhdan's Property (T)}
Property (T) is an analytic property of groups that can be seen as a rigidity phenomenon for unitary representations. It was introduced by Kazhdan \cite{Kaz} as a tool to prove that certain lattices in Lie groups are finitely generated, and it has since become one of the most important concepts in analytic group theory, providing connections with several branches of mathematics. 

Let $G$ be a (discrete) group. We say that $G$ has Property (T) if there is a finite subset $Q\subset G$ and a constant $c>0$ such that, for every unitary representation $\pi$ of $G$ on a Hilbert space $\mathcal{H}$,
\begin{align}\label{Kazh_prop}
\max_{s\in Q}\|\pi(s)\xi-\xi\|\geq c\|\xi\|,\quad\forall\xi\in (\mathcal{H}^\pi)^\perp,
\end{align}
where $(\mathcal{H}^\pi)^\perp$ denotes the orthogonal complement of the space of $\pi$-invariant vectors. Moreover, if this holds, $G$ is finitely generated, and any finite generating set $Q$ will satisfy \eqref{Kazh_prop} for some constant $c>0$. We refer the reader to \cite{BedlHVa} for a detailed account on Property (T) and its many characterisations and applications.

Property (T) is often regarded as a strong negation of amenability. A group $G$ is said to be amenable if the space $\ell^\infty(G)$ admits a left-invariant mean; we refer the reader to \cite{Jus} for details. Amenability is an obstruction to Property (T) in the following sense: if $G$ is amenable and satisfies Property (T), then it is necessarily finite; see e.g. \cite[Theorem 1.1.6]{BedlHVa}.

\subsection{Geometric Property (T)}
In connection to their work on the coarse Baum--Connes conjecture, Willett and Yu \cite{WiYu2} introduced Geometric Property (T) as an analogue of Kazhdan's Property (T) in the context of metric spaces. These ideas were further developed in \cite{WiYu}, where a new, equivalent definition of Geometric Property (T) was given. This definition is quite similar in spirit to the one for groups, but the absence of a group structure requires one to consider a different kind of representation, and a new notion of invariant vector.

We will focus on metric spaces $(X,d)$ such that the distance $d$ is allowed to take the value $\infty$. We will always assume that $(X,d)$ has \textit{bounded geometry}, meaning that, for every $R>0$,
\begin{align*}
\sup_{x\in X}|B(x,R)|<\infty,
\end{align*}
where $B(x,R)$ denotes the ball of radius $R$ centred at $x$. A typical example of such a space is the disjoint union of a family of finite graphs of uniformly bounded degree. We define the \textit{tube} of diameter $R$ as
\begin{align*}
\operatorname{Tube}_X(R)=\{(x,y)\in X\times X\ \mid\ d(x,y)\leq R\}.
\end{align*}
We say that a subset $E$ of $X\times X$ is \textit{controlled} if there is $R\in(0,\infty)$ such that $E\subseteq\operatorname{Tube}_X(R)$. Given two subsets $E,F\subset X\times X$, their \textit{composition} is defined as
\begin{align*}
E\circ F=\{(x,y)\in X\times X\ \mid\ \exists z\in X,\ (x,z)\in E,\ (z,y)\in F\}.
\end{align*}
We define inductively $E^{\circ (n+1)}=E\circ E^{\circ n}$, where $E^{\circ 1}=E$. We say that a controlled set $E$ is \textit{generating} if, for any controlled set $F$, there is $n\geq 1$ such that $F\subseteq E^{\circ n}$. We say that $X$ is \textit{monogenic} if there exists a controlled generating set inside $X$.

Two points $x,y$ are in the same \textit{coarse component} if the set $\{(x,y)\}$ is controlled, which is equivalent to the fact that $d(x,y)<\infty$. This is an equivalence relation and the equivalence classes are called \textit{coarse components} of $X$. The space is \textit{coarsely connected} if there is only one coarse component. One can show that a monogenic coarsely connected metric space is coarsely geodesic; see e.g. \cite[Proposition 2.57]{Roe}. We will use the same convention as in \cite{WiYu}:\\

\noindent\textbf{Convention:} Throughout this paper, a \textit{space} will mean a bounded geometry, monogenic metric space with at most countably many coarse components.\\

The \textit{translation algebra} $\C_u[X]$ is the set of matrices $T=(T_{x,y})_{x,y\in X}$ with coefficients in $\C$, satisfying
\begin{align*}
\sup_{x,y\in X}|T_{x,y}|<\infty,
\end{align*}
and such that the support
\begin{align*}
\operatorname{supp}(T)=\{(x,y)\in X\times X\ \mid\ T_{x,y}\neq 0\}
\end{align*}
is a controlled set. The usual matrix operations and adjoint make $\C_u[X]$ into a $\ast$-algebra. The elements of $\C_u[X]$ are also called \textit{finite propagation operators}. 

A \textit{partial translation} on $X$ is a bijective map $t:A\to B$, where $A, B$ are subsets of $X$, and such that the graph
\begin{align*}
\operatorname{graph}(t)=\{(x,y)\in B\times A\ \mid\ x=t(y)\}
\end{align*}
is a controlled set. Partial translations can be viewed as elements of $\C_u[X]$ in the following way. If $t:A\to B$ is a partial translation, we can identify it with the matrix $v\in\C_u[X]$ defined by
\begin{align}\label{t-->v}
v_{x,y}=\begin{cases}
1, & t(y)=x\\
0,& \text{otherwise}.
\end{cases}
\end{align}
Observe that $\operatorname{supp}(v)=\operatorname{graph}(t)$. We will also refer to $v$ as a partial translation.

A \textit{representation} of $\C_u[X]$ on a Hilbert space $\mathcal{H}$ is a unital $\ast$-homomorphism $\pi:\C_u[X]\to\mathcal{B}(\mathcal{H})$, where $\mathcal{B}(\mathcal{H})$ stands for the algebra of bounded operators on $\mathcal{H}$. We say that $\xi\in\mathcal{H}$ is an \textit{invariant vector} if
\begin{align*}
\pi(v)\xi=\pi(vv^*)\xi,
\end{align*}
for every partial translation $v$. We denote by $\mathcal{H}^\pi$ the subspace of invariant vectors of $\mathcal{H}$.

Now we can give the definition of Geometric Property (T) as in \cite[Definition 3.4]{WiYu}. We refer the reader to that same paragraph for the analogy with Property (T) for groups. We say that a space $X$ has Property (T) if, for every controlled generating set $E\subseteq X$, there exists $c>0$ such that, for every representation $\pi:\C_u[X]\to\mathcal{B}(\mathcal{H})$ and every $\xi\in(\mathcal{H}^\pi)^\perp$, there is a partial translation $v$ with support in $E$ satisfying
\begin{align*}
\|(\pi(v)-\pi(vv^*))\xi\|\geq c\|\xi\|.
\end{align*}

This property not only is inspired by the one for groups, it also provides a characterisation in the residually finite case. More precisely, a residually finite group has property (T) if and only if the box space associated to a nested family of finite-index subgroups with trivial intersection has Property (T); see \cite[Lemma 7.3]{WiYu2} or \cite[Theorem 7.1]{WiYu} for details.

The notion of amenability for spaces also plays a crucial role in this work. As in the case of groups, it is defined in terms of the existence of an invariant mean; see Section \ref{Ss_amen} for details. Moreover, for infinite coarsely connected spaces, non-amenability is equivalent to Property (T); see \cite[Corollary 6.1]{WiYu}. We will also give an alternative proof of this fact in Lemmas \ref{Lem_no_T->amen} and \ref{Lem_Kazh_infinite}.

\subsection{Main results}
For groups, Property (T) can be characterised by the existence of a certain projection in the full group C${}^*$-algebra; see \cite[Lemma 2]{AkeWal} or \cite[Proposition 2]{Val}. The main goal of this paper is to extend this result to spaces by looking at the maximal uniform Roe algebra instead.

Let $X$ be a space. The C${}^*$-algebra $C_{u,\max}^*(X)$ is defined as the completion of $\C_u[X]$ for the norm
\begin{align*}
\|T\|_{C_{u,\max}^*(X)}=\sup\left\{\|\pi(T)\|\ \mid\ \pi\text{ is a representation of }\C_u[X]\right\}.
\end{align*}
By \cite[Lemma 3.4]{GoWaYu}, this supremum is finite, so this is a well-defined norm. We call $C_{u,\max}^*(X)$ the \textit{maximal uniform Roe algebra} of $X$.

In order to precisely state our main result, we need to consider the following subalgebra of $\C_u[X]$:

\begin{align}\label{def_A_X}
\mathcal{A}_X=\left\{A\in\C_u[X] \,\mid\, \exists c\in\C,\ \forall x\in X,\ \sum_{y\in X}A_{x,y}=\sum_{y\in X}A_{y,x}=c\right\}.
\end{align}
This is a $\ast$-subalgebra of $\C_u[X]$, and the map $\theta_X:\mathcal{A}_X\to\C$ defined by
\begin{align}\label{def_theta}
\theta_X(A)=\sum_{y\in X}A_{x,y}
\end{align}
is a $\ast$-homomorphism. We will denote by $\mathcal{A}_X^+$ the subset of $\mathcal{A}_X$ given by all those $A$ such that $A_{x,y}\geq 0$ for all $x,y\in X$. We also define $\left(\mathcal{A}_X^+\right)_1=\mathcal{A}_X^+\cap\theta_X^{-1}(\{1\})$.

We may now state the characterisation of Property (T) for spaces. Recall that a projection in a C${}^*$-algebra is an element $P$ such that $P=P^*=P^2$.

\begin{thm}\label{Thm_char_T}
Let $X$ be a space. The following are equivalent:
\begin{itemize}
\item[i)] The space $X$ has Property (T).
\item[ii)] There exists a projection $P$ in $C_{u,\max}^*(X)$ such that, for every representation $\pi:\C_u[X]\to\mathcal{B}(\mathcal{H})$, $\pi(P)$ is the orthogonal projection onto the subspace of invariant vectors $\mathcal{H}^\pi$.
\end{itemize}
Moreover, in this case, $P$ lies in the closure of $\left(\mathcal{A}_X^+\right)_1$ inside $C_{u,\max}^*(X)$.
\end{thm}

We see from the definition of the norm of $C_{u,\max}^*(X)$ that, if the projection in Theorem \ref{Thm_char_T} exists, then it is unique. We call it the \textit{Kazhdan projection} of $C_{u,\max}^*(X)$.

Since the spaces that we consider here decompose into coarsely connected components, it is natural to ask what the Kazhdan projection looks like on each component. More precisely, let
\begin{align*}
X=\bigsqcup_{n}X_n
\end{align*}
be the decomposition of $X$ into coarsely connected components. Let $Q_n:\C_u[X]\to\C_u[X_n]$ be the restriction map to $X_n$. In other words,
\begin{align*}
Q_n(T)=T\mathds{1}_{X_n},\quad\forall T\in\C_u[X],
\end{align*}
where $\mathds{1}_{X_n}$ is the partial translation associated to the identity on $X_n$. This map extends to a $\ast$-homomorphism $Q_n:C_{u,\max}^*(X)\to C_{u,\max}^*(X_n)$; see Lemma \ref{Lem_Q}. Using this fact, together with Theorem \ref{Thm_char_T}, we prove the following.

\begin{cor}\label{Cor_ccc}
Let $X$ be a space, and let $X=\bigsqcup_{n}X_n$ be its decomposition into coarsely connected components. If $X$ has Property (T), then so does $X_n$ for every $n$. Moreover, if $P$ denotes the Kazhdan projection of $C_{u,\max}^*(X)$, then $Q_n(P)$ belongs to $\C_u[X_n]$ and it is the Kazhdan projection of $C_{u,\max}^*(X_n)$. Its coefficients are given by
\begin{align*}
Q_n(P)_{x,y}=\begin{cases}
\frac{1}{|X_n|} & \text{if } X_n \text{ is finite},\\
0 & \text{otherwise},
\end{cases}
\end{align*}
for every $x,y$ in $X_n$.
\end{cor}

This paper is organised as follows. In Section \ref{S_SC}, we prove the implication (ii)$\implies$(i) in Theorem \ref{Thm_char_T}. The converse is proved in Section \ref{S_NC}. Finally, we prove Corollary \ref{Cor_ccc} in Section \ref{S_cor}.

\section{{\bf Proof of sufficiency}}\label{S_SC}

In this section, we show that the existence of a Kazhdan projection implies Property (T). For this purpose, we need to review some facts about ultraproducts of representations, and the notion of amenability for spaces.

\subsection{Ultraproducts}
Ultraproducts provide a very powerful tool in functional analysis. In particular, for group representations, they allow one to construct invariant vectors as ultralimits of almost invariant vectors; see e.g. \cite[Proposition 1.1]{ChCoSt}. We briefly describe here how this idea can be extended to representations of $\C_u[X]$; for details on ultraproducts of spaces and operators, we refer the reader to \cite{Hei}.

Let $X$ be a space, and let $I$ be a directed set such that, for every $i\in I$ there is a representation $\pi_i:\C_u[X]\to\mathcal{B}(\mathcal{H}_i)$. We can consider the filter $\mathcal{F}$ on $I$ given by all the sets containing a subset of the form $\{\ i\in I\mid\ i\geq i_0\}$ for some $i_0\in I$. Now let $\mathcal{U}$ be an ultrafilter containing $\mathcal{F}$, and let
\begin{align*}
\mathcal{H}=(\mathcal{H}_i)_{\mathcal{U}}
\end{align*}
be the ultraproduct Hilbert space. We can define a representation $\pi:\C_u[X]\to\mathcal{B}(\mathcal{H})$ by
\begin{align*}
\pi(T)(\eta_i)_{\mathcal{U}}=(T\eta_i)_{\mathcal{U}},\quad\forall T\in\C_u[X],\ \forall(\eta_i)_{\mathcal{U}}\in\mathcal{H}.
\end{align*}
This construction allows us to obtain invariant vectors in the following way.

\begin{prop}\label{Prop_ultralim}
Let $X$ be a space, and let $\pi_i:\C_u[X]\to\mathcal{B}(\mathcal{H}_i)$ be a net of representations. Assume that there is a net of unit vectors $\xi_i\in\mathcal{H}_i$ such that, for every partial translation $v$,
\begin{align*}
\lim_i\|\pi_i(v)\xi_i-\pi_i(vv^*)\xi_i\|=0.
\end{align*}
Then, taking the ultraproduct $\mathcal{H}=(\mathcal{H}_i)_{\mathcal{U}}$ as above, the unit vector $\xi=(\xi_i)_{\mathcal{U}}$ belongs to $\mathcal{H}^\pi$.
\end{prop}
\begin{proof}
First observe that
\begin{align*}
\|\xi\|=\lim_{\mathcal{U}}\|\xi_i\|=1,
\end{align*}
so $\xi$ is indeed a unit vector. Moreover, for every partial translation $v$,
\begin{align*}
\|\pi(v)\xi-\pi(vv^*)\xi\|=\lim_{\mathcal{U}}\|\pi_i(v)\xi_i-\pi_i(vv^*)\xi_i\|=0,
\end{align*}
which shows that $\xi$ belongs to $\mathcal{H}^\pi$.
\end{proof}

\subsection{Amenability}\label{Ss_amen}
A space $X$ is said to be \textit{amenable} if $\ell^\infty(X)$ admits an invariant mean. More precisely, there is a positive unital linear functional $\phi:\ell^\infty(X)\to\C$ such that, for every $f\in\ell^\infty(X)$, and every partial translation $t:A\to B$ such that $B$ contains the support of $f$,
\begin{align*}
\phi(f\circ t)=\phi(f).
\end{align*}
Here we view $f\circ t$ as an element of $\ell^\infty(X)$ by extending it by $0$ outside of $A$. The following lemma will not be used in the proofs of our results. We state it here with the sole purpose of emphasising that, although this definition of amenability naturally extends the one for groups, it has some rather counter-intuitive properties.

\begin{lem}
Let $X$ be a space and let $X_0$ be a coarse component of $X$. If $X_0$ is amenable, then so is $X$.
\end{lem}
\begin{proof}
Let $\phi\in\ell^\infty(X_0)^*$ be an invariant mean. Define $\tilde{\phi}:\ell^\infty(X)\to\C$ by
\begin{align*}
\tilde{\phi}(f)=\phi(f|_{X_0}),\quad\forall f\in\ell^\infty(X),
\end{align*}
where $f|_{X_0}$ denotes the restriction of $f$ to $X_0$. Then $\tilde{\phi}$ is an invariant mean on $\ell^\infty(X)$.
\end{proof}

Observe that there is a copy of $\ell^{\infty}(X)$ inside $\C_u[X]$ given by the subalgebra of diagonal matrices. The following lemma is a reformulation of \cite[Proposition 6.1]{WiYu}.

\begin{lem}\label{Lem_char_am}
Let $X$ be a space. The following are equivalent:
\begin{itemize}
\item[(i)] The space $X$ is amenable.
\item[(ii)] There exists a representation $\pi:\C_u[X]\to\mathcal{B}(\mathcal{H})$ such that $\mathcal{H}^\pi\neq\{0\}$.
\end{itemize}
\end{lem}
\begin{proof}
Assume first that (ii) holds, and let $\xi$ be a unit vector in $\mathcal{H}^\pi$. We define a unital linear map $\phi:\C_u[X]\to\C$ by
\begin{align*}
\phi(T)=\langle\pi(T)\xi,\xi\rangle,\quad\forall T\in\C_u[X].
\end{align*}
We claim that the restriction of $\phi$ to $\ell^\infty(X)$ is an invariant mean. Indeed, let $f\in\ell^\infty(X)$, $t:A\to B$ a partial translation such that $B$ contains the support of $f$, and let $v$ be the element of $\C_u[X]$ associated to $t$ as in \eqref{t-->v}. Recall that we view $f\circ t$ as an element of $\ell^\infty(X)$ by extending it by $0$ outside of $A$. Observing that $f\circ t=v^*fv$ and $vv^*=\mathds{1}_B$, we see that
\begin{align*}
\phi(f\circ t)&=\langle\pi(v^*fv)\xi,\xi\rangle\\
&=\langle\pi(fv)\xi,\pi(v)\xi\rangle\\
&=\langle\pi(f)\pi(vv^*)\xi,\pi(vv^*)\xi\rangle\\
&=\langle\pi(f)\xi,\xi\rangle\\
&=\phi(f).
\end{align*}
Finally, the identity $\phi(|f|^2)=\|\pi(f)\xi\|^2$ shows that $\phi$ is positive. We conclude that $\phi$ is an invariant mean.\\
Conversely, assume that (i) holds. By \cite[Lemma 6.1]{WiYu}, there is a net of unit vectors $(\xi_i)_{i\in I}$ in $\ell^2(X)$ such that, for every partial translation $v$,
\begin{align*}
\|v\xi_i-vv^*\xi_i\|\to 0.
\end{align*}
By Proposition \ref{Prop_ultralim}, the ultraproduct representation has a non-trivial invariant vector.
\end{proof}

In \cite[Corollary 6.1]{WiYu}, it was proved that, for infinite coarsely connected spaces, Property (T) is equivalent to non-amenability. Although the coarse connectedness assumption is needed for one direction of the equivalence, the other one holds in full generality. In order to prove this, we will use the following characterisation of Property (T); for a proof see \cite[Proposition 3.1]{WiYu}.

\begin{lem}[Willett--Yu]\label{Lem_forall_exists}
Let $X$ be a space. The following are equivalent:
\begin{itemize}
\item[(i)] The space $X$ has Property (T).
\item[(ii)] There exist a controlled generating set $E\subseteq X$ and $c>0$ such that, for every representation $\pi:\C_u[X]\to\mathcal{B}(\mathcal{H})$ and every $\xi\in(\mathcal{H}^\pi)^\perp$, there is a partial translation $v$ with support in $E$ satisfying
\begin{align*}
\|(\pi(v)-\pi(vv^*))\xi\|\geq c\|\xi\|.
\end{align*}
\end{itemize}
\end{lem}

\begin{rmk}
Lemma \ref{Lem_forall_exists} says that it suffices to consider one generating set in the definition of Property (T).
\end{rmk}

\begin{lem}\label{Lem_no_T->amen}
Let $X$ be a space. If $X$ does not have Property (T), then it is amenable.
\end{lem}
\begin{proof}
By Lemma \ref{Lem_forall_exists}, for every $n\geq 1$, there is a representation $\pi_n:\C_u[X]\to\mathcal{B}(\mathcal{H}_n)$ and a unit vector $\xi_n\in(\mathcal{H}_n^{\pi_n})^\perp$ such that, for every partial translation $v$ with support in $\operatorname{Tube}_X(n)$,
\begin{align*}
\|(\pi_n(v)-\pi_n(vv^*))\xi_n\|< \frac{1}{n}.
\end{align*}
By Proposition \ref{Prop_ultralim}, the ultraproduct representation has a non-trivial invariant vector. Therefore, by Lemma \ref{Lem_char_am}, $X$ is amenable.
\end{proof}

\subsection{Proof of Property (T)}
Now we are ready to prove that the existence of the Kazhdan projection implies Property (T).

\begin{lem}\label{Lem_SC2}
Let $X$ be a space, and assume that there is $P$ in $C_{u,\max}^*(X)$ such that, for every representation $\pi:\C_u[X]\to\mathcal{B}(\mathcal{H})$, $\pi(P)$ is the orthogonal projection onto the subspace of invariant vectors $\mathcal{H}^\pi$. Then $X$ has Property (T).
\end{lem}
\begin{proof}
Assume by contradiction that $X$ does not have Property (T). Thus, by Lemma \ref{Lem_forall_exists}, for every $n\geq 1$, there is a representation $\pi_n:\C_u[X]\to\mathcal{B}(\mathcal{H}_n)$ and a unit vector $\xi_n\in(\mathcal{H}_n^{\pi_n})^\perp$ such that, for every partial translation $v$ with support in $\operatorname{Tube}_X(n)$,
\begin{align*}
\|(\pi_n(v)-\pi_n(vv^*))\xi_n\|< \frac{1}{n}.
\end{align*}
Now we take the ultraproduct $\mathcal{H}=(\mathcal{H}_n)_{\mathcal{U}}$ for some ultrafilter $\mathcal{U}$ on $\N$. By Proposition \ref{Prop_ultralim}, $\xi=(\xi_n)_{\mathcal{U}}$ is a non-trivial invariant vector for the ultraproduct representation. However, since $\xi_n$ belongs to $(\mathcal{H}_n^{\pi_n})^\perp$, we have
\begin{align*}
1=\|\pi(P)\xi\|=\lim_{\mathcal{U}}\|\pi_n(P)\xi_n\|=0,
\end{align*}
which is a contradiction.
\end{proof}

\section{{\bf Proof of necessity}}\label{S_NC}

For the proof of the implication (i)$\implies$(ii) in Theorem \ref{Thm_char_T}, we will need some facts about colourings on graphs. 

\subsection{Graphs and colourings}
Let $X=(V,E)$ be a graph and $d\geq 1$ an integer. We say that a map $c:E\to\{1,\ldots,d\}$ is a $d$-colouring of $E$ if, whenever $e_1, e_2$ are two adjacent edges, we have $c(e_1)\neq c(e_2)$. The following result was proved in \cite{Viz}; see also \cite[Theorem 5.3.2]{Die}.

\begin{thm}[Vizing]\label{Thm_Viz}
Let $d\in\N$ and let $X=(V,E)$ be a finite graph such that every vertex in $X$ has degree at most $d$. Then there exists a $(d+1)$-colouring of $E$.
\end{thm}

This result extends to infinite graphs by standard arguments. We include the proof of this fact for completeness.

\begin{cor}\label{Cor_Viz}
Let $d\in\N$ and let $X=(V,E)$ be a countable infinite graph such that every vertex in $X$ has degree at most $d$. Then there exists a $(d+1)$-colouring of $E$.
\end{cor}
\begin{proof}
Let $(X_n)$ be an increasing sequence of finite subgraphs of $X$ such that $X=\bigcup_n X_n$. Let us write $X_n=(V_n,E_n)$. Since $E_n$ is contained in $E$, every $v\in V_n$ has degree at most $d$ in $X_n$. By Theorem \ref{Thm_Viz}, $E_n$ admits a $(d+1)$-colouring. Let
\begin{align*}
K=\{1,\ldots,d+1\}^E.
\end{align*}
Recall that this is a compact space for the product topology. For each $n$, define
\begin{align*}
K_n=\left\{c\in K\ \mid\ c|_{E_n} \text{ is a } (d+1)\text{-colouring}\right\}.
\end{align*}
Then $(K_n)$ is a decreasing sequence of closed subsets of $K$. Moreover, they are non-empty by the observation above. By the finite intersection property, the set $\bigcap_n K_n$ is non-empty. Any element of this intersection is a $(d+1)$-colouring of $E$.
\end{proof}

Let now $X$ be a space. We denote by $\mathcal{S}_X$ the set of finite propagation permutation matrices. In other words, every $A$ in $\mathcal{S}_X$ is an element of $(\mathcal{A}_X^+)_1$ with coefficients in $\{0,1\}$.

\begin{rmk}\label{Rmk_perm_unit}
Since every $A\in\mathcal{S}_X$ is a permutation matrix, it is in particular a unitary. Indeed, for every $x\in X$, there is a unique $y\in X$ such that $A_{x,y}=1$, and vice versa. All the other coefficients are $0$. Hence
\begin{align*}
(A^*A)_{x,y}=\sum_{z\in X}A_{z,x}A_{z,y}
=\begin{cases}
1, &\text{if } x=y,\\
0, &\text{otherwise.}
\end{cases}
\end{align*}
The same reasoning applies to $AA^*$.
\end{rmk}

\begin{lem}\label{Lem_dec_v}
Let $X$ be a space and let $R>0$. There exist $A_0,A_1,\ldots,A_n\in\mathcal{S}_X$ such that, for every partial translation $v$ with support in $\operatorname{Tube}_X(R)$, there exist $f_0,f_1,\ldots,f_n\in\ell^\infty(X)$ with $f_i^2=f_i$ such that
\begin{align*}
v=\sum_{i=0}^n f_iA_i.
\end{align*}
Moreover, we have $vv^*=\sum_{i=0}^nf_i$.
\end{lem}
\begin{proof}
Define a graph $(V,E)$ with $V=X$ and
\begin{align*}
E=\left\{ \{x,y\}\ \mid\ x\neq y,\ d(x,y)\leq R\right\}.
\end{align*}
It has bounded degree because $X$ has bounded geometry. By Corollary \ref{Cor_Viz}, there is a colouring $c:E\to\{1,\ldots,n\}$. We will extend $c$ to the diagonal of $X\times X$ by
\begin{align*}
c(\{x,x\})=0,\quad\forall x\in X.
\end{align*}
Define $A_0=1$ and $A_i$ ($i=1,\ldots,n$) by
\begin{align*}
(A_i)_{x,y}=\begin{cases}
1 & \text{if } \{x,y\}\in c^{-1}(\{i\}),\\
1 & \text{if } x=y,\ \nexists\, z\text{ such that } \{x,z\}\in c^{-1}(\{i\}),\\
0 & \text{otherwise}.
\end{cases}
\end{align*}
By definition, every coefficient of $A_i$ is either $0$ or $1$. Moreover, since $c$ is a colouring, for every $x\in X$, there is a unique $y\in X$ such that $(A_i)_{x,y}=1$, and vice versa. Thus $A_i$ is an element of $\mathcal{S}_X$. Now, given a partial translation $v$ with support in $\operatorname{Tube}_X(R)$, we define $f_i$ ($i=0,\ldots,n$) by
\begin{align*}
f_i(x)=\begin{cases}
1 & \text{if }\exists\, y\text{ such that }\{x,y\}\in c^{-1}(\{i\})\text{ and } v_{x,y}=1,\\
0 & \text{otherwise}.
\end{cases}
\end{align*}
Then we have, for every $i\in\{0,\ldots,n\}$,
\begin{align*}
f_i(x)(A_i)_{x,y}=\begin{cases}
1 & \text{if }\{x,y\}\in c^{-1}(\{i\})\text{ and } v_{x,y}=1,\\
0 & \text{otherwise},
\end{cases}
\end{align*}
which shows that
\begin{align*}
v=\sum_{i=0}^n f_iA_i.
\end{align*}
Moreover,
\begin{align*}
\sum_{i=0}^n f_i(x)=\begin{cases}
1 & \text{if }\exists y,\ v_{x,y}=1,\\
0 & \text{otherwise},
\end{cases}
\end{align*}
which means that $vv^*=\sum_{i=0}^nf_i$.
\end{proof}

\subsection{Proof of the necessary condition}

Now we will prove the implication (i)$\implies$(ii) in Theorem \ref{Thm_char_T}. For this purpose, we will need a characterisation of Property (T) in terms of diagonal operators. Recall that there is a copy of $\ell^{\infty}(X)$ in $\C_u[X]$ given by the subalgebra of diagonal matrices. We can define a map $\Phi_X:\C_u[X]\to\ell^\infty(X)$ by
\begin{align*}
\Phi_X(T)(x)=\sum_{y\in X}T_{x,y},\quad\forall T\in\C_u[X],\ \forall x\in X.
\end{align*}
Observe that $\Phi_X$ is well defined because $X$ has bounded geometry. The following was proved in \cite[Lemma 3.1]{WiYu}.

\begin{lem}[Willett--Yu]\label{Lem_WY}
Let $X$ be a space. For every representation $\pi:\C_u[X]\to\mathcal{B}(\mathcal{H})$, a vector $\xi\in\mathcal{H}$ is invariant if and only if
\begin{align*}
\pi(\Phi_X(T))\xi=\pi(T)\xi,
\end{align*}
for all $T\in\C_u[X]$.
\end{lem}

\begin{rmk}\label{Rmk_Phi_theta}
Recall the definition of $\mathcal{A}_X$ in \eqref{def_A_X}. The image of $\mathcal{A}_X$ under $\Phi_X$ consists of the constant functions on $X$. Moreover, for all $A\in\mathcal{A}_X$ and $x\in X$,
\begin{align}\label{Phi=theta}
\Phi_X(A)(x)=\theta_X(A),
\end{align}
where $\theta_X$ is the morphism defined in \eqref{def_theta}.
\end{rmk}

\begin{lem}\label{Lem_NC}
Let $X$ be a space with Property (T). Then there exists a projection $P$ in the closure of $(\mathcal{A}_X^+)_1$ inside $C_{u,\max}^*(X)$ such that, for every representation $\pi:\C_u[X]\to\mathcal{B}(\mathcal{H})$, $\pi(P)$ is the orthogonal projection onto the subspace of invariant vectors $\mathcal{H}^\pi$.
\end{lem}
\begin{proof}
Let $E$ be a controlled generating set. By definition, there is a constant $c>0$ such that, for every representation $\pi:\C_u[X]\to\mathcal{B}(\mathcal{H})$ and every $\xi\in(\mathcal{H}^\pi)^\perp$, there is a partial translation $v$ with support in $E$ such that
\begin{align}\label{ineq_T}
\|\pi(v)\xi-\pi(vv^*)\xi\|\geq c\|\xi\|.
\end{align}
By Lemma \ref{Lem_dec_v}, there exist $A_0,A_1,\ldots,A_n\in\mathcal{S}_X$ such that, for every such $v$, there exist $f_0,f_1,\ldots,f_n\in\ell^\infty(X)$ satisfying $f_i^2=f_i$ and
\begin{align*}
v=\sum_{i=0}^n f_iA_i.
\end{align*}
Since $\pi(vv^*)=\sum\pi(f_i)$ and $\|\pi(f_i)\|^2=\|\pi(f_i)\|$, the inequality \eqref{ineq_T} yields
\begin{align*}
c\|\xi\| &\leq \left\|\sum_{i=1}^n\pi(f_i)(\pi(A_i)\xi-\xi)\right\|\\
&\leq \sum_{i=1}^n\|\pi(A_i)\xi-\xi\|.
\end{align*}
This shows that there exists $j$ in $\{1,\ldots,n\}$ such that
\begin{align*}
\|\pi(A_j)\xi-\xi\|\geq \frac{c}{n}\|\xi\|.
\end{align*}
Recall that $A_j$ is a unitary matrix in $\C_u[X]$; see Remark \ref{Rmk_perm_unit}. Hence $\pi(A_j)$ is a unitary operator. By the parallelogram identity,
\begin{align*}
\left\|\frac{\pi(A_j)\xi+\xi}{2}\right\|\leq \delta\|\xi\|,
\end{align*}
where $\delta=\left(1-\left(\frac{c}{2n}\right)^2\right)^{1/2}$. Notice that $\delta$ only depends on the controlled set $E$, and not on the representation $\pi$ or the vector $\xi$. Defining $A\in\left(\mathcal{A}_X^+\right)_1$ by
\begin{align*}
A=\frac{1}{n}\sum_{i=1}^n\frac{1+A_i}{2},
\end{align*}
we see that
\begin{align*}
\|\pi(A)\xi\| &\leq \frac{1}{n}\left\|\frac{\xi+\pi(A_j)\xi}{2}\right\| + \frac{1}{n}\sum_{i\neq j}\left\|\frac{\xi+\pi(A_i)\xi}{2}\right\|\\
&\leq \frac{\delta}{n}\|\xi\| + \frac{n-1}{n}\|\xi\|\\
&= \tilde{\delta}\|\xi\|,
\end{align*}
where $\tilde{\delta}=1-\frac{1-\delta}{n}<1$. This inequality holds for every representation $\pi$ and every $\xi$ in $(\mathcal{H}^\pi)^\perp$. Now let $p_\pi$ be the orthogonal projection onto $\mathcal{H}^\pi$. Observe that $\pi(A)p_\pi=p_\pi$ because $\theta_X(A)=1$; see Lemma \ref{Lem_WY} and Remark \ref{Rmk_Phi_theta}. On the other hand,
\begin{align*}
p_\pi\pi(A)=\left(\pi(A^*)p_\pi\right)^*=p_\pi^*=p_\pi.
\end{align*}
Using these two facts, an induction argument shows that
\begin{align*}
\pi(A^k)-p_\pi=(\pi(A)-p_\pi)^k
\end{align*}
for all $k\geq 1$. Let $\eta\in\mathcal{H}$. By the discussion above, we know that
\begin{align*}
\|\pi(A)\eta-p_\pi\eta\| &= \|\pi(A)(\eta-p_\pi\eta)\|\\
&\leq \tilde{\delta}\|\eta-p_\pi\eta\|\\
&\leq \tilde{\delta}\|\eta\|.
\end{align*}
In other words, $\|\pi(A)-p_\pi\|\leq\tilde{\delta}$. Furthermore, for all $k\geq 1$,
\begin{align}\label{ineq_A^k}
\left\|\pi(A^k)-p_\pi\right\| = \left\|(\pi(A)-p_\pi)^k\right\| \leq\tilde{\delta}^k.
\end{align}
Since this bound holds for every representation, we see that $(A^k)_{k\geq 1}$ is a Cauchy sequence in $C_{u,\max}^*(X)$, and therefore it converges to an element $P$. Since $A^k$ belongs to $\left(\mathcal{A}_X^+\right)_1$ for every $k\geq 1$, $P$ is in the closure of $\left(\mathcal{A}_X^+\right)_1$. The inequality \eqref{ineq_A^k} shows that $\pi(P)=p_\pi$ for every representation $\pi$, which in turn shows that $P$ is itself a projection.
\end{proof}

\begin{proof}[Proof of Theorem \ref{Thm_char_T}]
The proof follows directly from Lemmas \ref{Lem_SC2} and \ref{Lem_NC}.
\end{proof}

\section{{\bf Kazhdan projections and coarsely connected components}}\label{S_cor}

In this section, we study coarsely connected spaces and prove Corollary \ref{Cor_ccc}. Most of the results concerning Property (T) that we present here were already established in \cite{WiYu}. Our main contribution is the description of the Kazhdan projection in each case. We begin by characterising the Kazhdan projection of a finite space.

\begin{lem}\label{Lem_Kazh_finite}
Let $X$ be a finite, coarsely connected space. Then $X$ has Property (T), and its Kazhdan projection $P$ is given by
\begin{align*}
P_{x,y}=\frac{1}{|X|},\quad\forall x,y\in X.
\end{align*}
\end{lem}
\begin{proof}
Let $P\in\C_u[X]$ be given by $P_{x,y}=|X|^{-1}$ for all $x,y\in X$. Observe that $P$ belongs to $(\mathcal{A}_X^+)_1$, and $P=P^*=P^2$. By Theorem \ref{Thm_char_T}, we only need to show that $\pi(P)$ is the projection onto $\mathcal{H}^\pi$ for every representation $\pi:\C_u[X]\to\mathcal{B}(\mathcal{H})$. First, if $\xi$ belongs to $\mathcal{H}^\pi$, by Lemma \ref{Lem_WY} we have
\begin{align*}
\pi(P)\xi=\pi(\Phi_X(P))\xi=\pi(1)\xi=\xi.
\end{align*}
On the other hand, for all $T\in\C_u[X]$ and $x,y\in X$,
\begin{align*}
(TP)_{x,y}=\sum_{z\in X}T_{x,z}|X|^{-1}=\Phi_X(T)_{x,x}|X|^{-1}=\sum_{z\in X}\Phi_X(T)_{x,z}P_{z,y}=(\Phi_X(T)P)_{x,y}.
\end{align*}
Hence, for every $\xi\in\mathcal{H}$,
\begin{align*}
\pi(T)\pi(P)\xi=\pi(\Phi_X(T))\pi(P)\xi,
\end{align*}
which shows that $\pi(P)\xi$ is an invariant vector. We conclude that $\pi(P)$ is the orthogonal projection onto $\mathcal{H}^\pi$.
\end{proof}

Lemma \ref{Lem_Kazh_finite} settles the finite case. The infinite case is even simpler. The following was proved in \cite[Corollary 6.1]{WiYu}. We present here an alternative proof using Proposition \ref{Prop_ultralim}.

\begin{lem}\label{Lem_Kazh_infinite}
Let $X$ be an infinite, coarsely connected space. If $X$ has Property (T), then it is not amenable. In particular, the Kazhdan projection of $C_{u,\max}^*(X)$ is $0$.
\end{lem}
\begin{proof}
Assume by contradiction that $X$ is amenable. By \cite[Lemma 6.1]{WiYu}, there is a net of unit vectors $(\xi_i)_{i\in I}$ in $\ell^2(X)$ such that, for every partial translation $v$,
\begin{align*}
\|v\xi_i-vv^*\xi_i\|\to 0.
\end{align*}
Let $\lambda:C_{u,\max}^*(X)\to\mathcal{B}(\ell^2(X))$ be the extension of the natural representation of $\C_u[X]$ on $\ell^2(X)$. Let $\mathcal{U}$ be an ultrafilter on $I$, and let $\pi$ denote the ultraproduct representation on $\mathcal{H}=(\ell^2(X))_{\mathcal{U}}$. By Proposition \ref{Prop_ultralim}, $\xi=(\xi_i)_{\mathcal{U}}$ is a non-trivial invariant vector for $\pi$. In other words, if $P$ denotes the Kazhdan projection of $C_{u,\max}^*(X)$, then $\pi(P)\xi=\xi\neq 0$. On the other hand, since $X$ is infinite and coarsely connected, $\ell^2(X)$ does not have non-trivial invariant vectors. In particular, $\lambda(P)\xi_i=0$. This shows that
\begin{align*}
\|\xi\|=\|\pi(P)\xi\|=\lim_{\mathcal{U}}\|\lambda(P)\xi_i\|=0,
\end{align*}
which is a contradiction. We conclude that $X$ is not amenable. Therefore, by Lemma \ref{Lem_char_am}, $P$ is $0$.
\end{proof}

Now we have completely characterised Kazhdan projections of coarsely connected spaces. In order to prove Corollary \ref{Cor_ccc}, we need to study the restrictions of a Kazhdan projection on a space to its coarse components.

\begin{lem}\label{Lem_Q}
Let $X$ be a space, and let $X_0$ be a coarse component of $X$. The map $Q:\C_u[X]\to\C_u[X_0]$ defined by
\begin{align*}
T\mapsto T\mathds{1}_{X_0},\quad\forall T\in\C_u[X],
\end{align*}
where $\mathds{1}_{X_0}$ is the indicator function of $X_0$ seen as a diagonal operator, extends to a unital $\ast$-homomorphism $Q:C_{u,\max}^*(X)\to C_{u,\max}^*(X_0)$.
\end{lem}
\begin{proof}
Observe that $Q$ is a unital $\ast$-homomorphism from $\C_u[X]$ to $\C_u[X_0]$. We only need to show that it extends to a continuous map from $C_{u,\max}^*(X)$ to $C_{u,\max}^*(X_0)$. Let $T\in\C_u[X]$, and let us write it by blocks as follows:
\begin{align*}
T=T_0\oplus T',
\end{align*}
where $T_0\in\C_u[X_0]$ and $T'\in\C_u[X\setminus X_0]$. Observe that $Q(T)=T_0$. Let us now take a representation $\pi:\C_u[X_0]\to\mathcal{B}(\mathcal{H})$, and let $\tilde{\pi}:\C_u[X]\to\mathcal{B}(\mathcal{H}\oplus\ell^2(X\setminus X_0))$ be given by
\begin{align*}
\tilde{\pi}(T_0\oplus T')=\pi(T_0)\oplus T'.
\end{align*}
Then
\begin{align*}
\|\pi(T_0)\|_{\mathcal{B}(\mathcal{H})}
&\leq \max\left\{\|\pi(T_0)\|_{\mathcal{B}(\mathcal{H})},\|T'\|_{\mathcal{B}(\ell^2(X\setminus X_0))}\right\}\\
&=\|\tilde{\pi}(T_0\oplus T')\|_{\mathcal{B}(\mathcal{H}\oplus\ell^2(X\setminus X_0))}\\
&\leq \|T_0\oplus T'\|_{C_{u,\max}^*(X)}.
\end{align*}
Taking the supremum over all the representations of $\C_u[X_0]$, we find
\begin{align*}
\|Q(T)\|_{C_{u,\max}^*(X_0)}\leq\|T\|_{C_{u,\max}^*(X)},
\end{align*}
which completes the proof.
\end{proof}

\begin{lem}\label{Lem_Q(P)}
Let $X$ be a space, $X_0$ a coarse component of $X$, and $Q$ the map defined in Lemma \ref{Lem_Q}. If $X$ has Property (T), then so does $X_0$. Moreover, if $P$ is the Kazhdan projection of $C_{u,\max}^*(X)$, then $Q(P)$ is the Kazhdan projection of $C_{u,\max}^*(X_0)$.
\end{lem}
\begin{proof}
Since $Q$ is a $\ast$-homomorphism, $Q(P)$ is a projection. Now let $\pi:\C_u[X_0]\to\mathcal{B}(\mathcal{H})$ be a representation, and define $\tilde{\pi}:\C_u[X]\to\mathcal{B}(\mathcal{H})$ by $\tilde{\pi}=\pi\circ Q$. Again, since $Q$ is a $\ast$-homomorphism, $\tilde{\pi}$ is indeed a representation. Furthermore, $\tilde{\pi}(P)$ is the orthogonal projection onto the subspace of invariant vectors $\mathcal{H}^{\tilde{\pi}}$. Since $\tilde{\pi}(P)=\pi(Q(P))$, we only need to show that $\mathcal{H}^{\tilde{\pi}}=\mathcal{H}^{\pi}$. This is true because
\begin{align*}
\xi\in\mathcal{H}^{\tilde{\pi}} &\iff \tilde{\pi}(T)\xi=\tilde{\pi}(\Phi_X(T))\xi,\quad\forall T\in\C_u[X]\\
&\iff \pi(Q(T))\xi=\pi(Q(\Phi_X(T)))\xi,\quad\forall T\in\C_u[X]\\
&\iff \pi(T_0)\xi=\pi(\Phi_{X_0}(T_0))\xi,\quad\forall T_0\in\C_u[X_0]\\
&\iff \xi\in\mathcal{H}^{\pi}.
\end{align*}
\end{proof}

Now we are ready to prove Corollary \ref{Cor_ccc}.

\begin{proof}[Proof of Corollary \ref{Cor_ccc}]
The fact that each coarse component $X_n$ has Property (T) follows from Lemma \ref{Lem_Q(P)}. Moreover, Lemma \ref{Lem_Q(P)} also implies that $Q_n(P)$ is the Kazhdan projection of $C_{u,\max}^*(X_n)$. The description of $Q_n(P)$ in each case follows from Lemmas \ref{Lem_Kazh_finite} and \ref{Lem_Kazh_infinite}.
\end{proof}

\subsection*{Acknowledgements}
This work originated from a series of very enlightening discussions that I had with Mikael de la Salle in Lyon during a research visit in February 2023. I am very grateful to him for his inspiring ideas, and to the members of the \textit{Institut Camille Jordan} for their hospitality. I also thank Rufus Willett for his very valuable feedback. Finally, I thank the anonymous referee, whose comments and suggestions led to several improvements in the contents and presentation of this paper.

\bibliographystyle{plain} 

\bibliography{Bibliography}

\end{document}